\pdfoutput=1
\documentclass[11pt,reqno]{amsart}
\usepackage[T1]{fontenc}
\usepackage[utf8]{inputenc}
\usepackage{lmodern}
\usepackage{amsmath,amssymb,mathtools}
\usepackage[margin=1.05in]{geometry}
\usepackage{microtype}
\usepackage{booktabs,array,longtable}
\usepackage{enumitem}
\usepackage{url}
\usepackage[hidelinks]{hyperref}
\hypersetup{pdftitle={The Burr-Erdos-Graham-Sos conjecture for the seven-cycle},pdfauthor={Asad Shahab},pdfsubject={Erdos Problem 809; rainbow odd cycles}}
\numberwithin{equation}{section}
\newtheorem{theorem}{Theorem}[section]
\newtheorem{lemma}[theorem]{Lemma}
\newtheorem{proposition}[theorem]{Proposition}
\newtheorem{corollary}[theorem]{Corollary}
\theoremstyle{definition}
\newtheorem{definition}[theorem]{Definition}
\theoremstyle{remark}

\newcommand{\E}{\mathbb E}
\newcommand{\Prob}{\mathbb P}
\newcommand{\one}{\mathbf 1}
\newcommand{\MC}{\operatorname{mc}}

\newcommand{\cL}{\mathcal L}
\newcommand{\set}[1]{\{#1\}}
\newcommand{\abs}[1]{\lvert#1\rvert}

\newcommand{\ind}[1]{\one_{\{#1\}}}

\newcommand{\code}[1]{\texttt{\detokenize{#1}}}
\setlist[enumerate]{label=\textup{(\roman*)},leftmargin=2em,itemsep=2pt,topsep=4pt}
\allowdisplaybreaks[1]
\title[The Burr--Erd\H{o}s--Graham--S\'os conjecture for the seven-cycle]{The Burr--Erd\H{o}s--Graham--S\'os conjecture for the seven-cycle}
\author[The Burr--Erd\H{o}s--Graham--S\'os conjecture for the seven-cycle]{Asad Shahab}
\date{September 26, 2026}
\begin{document}
\begin{abstract}
For a graph $H$, let $f(n,e,H)$ be the least number of colors in an edge-coloring of some $n$-vertex graph with at least $e$ edges in which every copy of $H$ is rainbow. Burr, Erd\H{o}s, Graham, and S\'os conjectured that $f(n,\lfloor n^2/4\rfloor+1,C_{2k+1})=(1/8+o(1))n^2$ for every fixed $k\ge3$, and Buci\'c, Chen, and Ma recently proved this for all $k\ge4$. We prove the remaining case $k=3$:
\[
 f\left(n,\left\lfloor n^2/4\right\rfloor+1,C_7\right)
 =\left(\frac18+o(1)\right)n^2.
\]
The lower bound rests on a weighted palette inequality, which we prove with an exact rational certificate on five sampled vertices. Its main ingredients are a fractional matching of compatible triangular edges and private resources attached to nontriangular edges. A stable form of the inequality, combined with regularity, triangle removal, and a direct argument for graphs close to bipartite, transfers the bound to arbitrary edge-colorings. We also describe a Lean~4 formalization of the conjecture for every fixed $k\ge3$, which combines the new seven-cycle proof with a formalization of the Buci\'c--Chen--Ma argument for $k\ge4$.
\end{abstract}
\maketitle
\section{Introduction}\label{sec:introduction}
A subgraph of an edge-colored graph is \emph{rainbow} if its edges have pairwise distinct colors. Burr, Erd\H{o}s, Graham, and S\'os~\cite{BEGS1989} asked how few colors are needed to make every copy of a fixed graph $H$ rainbow. For a finite simple graph $G$, let $r_H(G)$ denote this minimum and, following Buci\'c, Chen, and Ma~\cite{BCM2026}, define
\begin{equation}\label{eq:fdefinition}
 f(n,e,H)=\min\set{r_H(G):\abs{V(G)}=n,\ e(G)\ge e},
\end{equation}
whenever the family of graphs is nonempty. Copies of $H$ are subgraphs, not necessarily induced; thus a copy of a cycle has distinct vertices but may have chords in $G$. Colorings need not be proper, and only colors that actually occur on an edge are counted. In~\cite{BEGS1989} the same quantity is denoted $\chi_S(n,e,H)$ and the host is required to have exactly $e$ edges. The two definitions agree, because deleting edges preserves the rainbow property and cannot increase the number of colors.

For odd cycles the natural threshold is $e=\lfloor n^2/4\rfloor+1$, one more than the Tur\'an number of $C_{2k+1}$ for large $n$, and there the answer depends strongly on the length of the cycle. For triangles the value is $3$, and for $C_5$ it is $\lfloor n/2\rfloor+3$ for large $n$, a result of Erd\H{o}s and Simonovits (see~\cite{BEGS1989,BCM2026}). For longer cycles, Burr, Erd\H{o}s, Graham, and S\'os proved that $f(n,\lfloor n^2/4\rfloor+1,C_{2k+1})$ is at least a positive multiple of $n^2$ for every fixed $k\ge3$~\cite[Theorem~5.1]{BEGS1989}. They conjectured that the correct constant is $1/8$, the value given by two disjoint cliques of nearly equal size. This conjecture is Erd\H{o}s Problem~\#809. Buci\'c, Chen, and Ma~\cite{BCM2026} proved it for every $k\ge4$. We settle the remaining case $k=3$.

\begin{theorem}\label{thm:main}
As $n\to\infty$ through all positive integers,
\begin{equation}\label{eq:main}
 f\left(n,\left\lfloor n^2/4\right\rfloor+1,C_7\right)
 =\left(\frac18+o(1)\right)n^2.
\end{equation}
\end{theorem}

The upper bound comes from the two-clique construction. The new content is the lower bound: for every $\varepsilon>0$ and all sufficiently large $n$, every $n$-vertex graph with at least $\lfloor n^2/4\rfloor+1$ edges needs at least $(1/8-\varepsilon)n^2$ colors in any edge-coloring in which every $C_7$ is rainbow.

\begin{corollary}\label{cor:allodd}
For every fixed integer $k\ge3$,
\begin{equation}\label{eq:allodd}
 f\left(n,\left\lfloor n^2/4\right\rfloor+1,C_{2k+1}\right)
 =\left(\frac18+o(1)\right)n^2.
\end{equation}
The case $k=3$ is Theorem~\ref{thm:main}; the cases $k\ge4$ are due to Buci\'c, Chen, and Ma.
\end{corollary}
\begin{proof}
For $k\ge4$, apply~\cite[Theorem~1.2]{BCM2026} at $e=\lfloor n^2/4\rfloor+1$; its square-root term is $O(n)$. The remaining case is Theorem~\ref{thm:main}.
\end{proof}

\subsection{Outline of the proof}
The \emph{conflict graph} of $G$ has vertex set $E(G)$, two edges being adjacent when they lie on a common simple $C_7$; its chromatic number is $r_{C_7}(G)$. Rather than work with this graph directly, we pass to a vertex-weighted graph and replace common seven-cycles by a condition on walks: two edges are \emph{compatible} if their endpoints are not joined by complementary walks of lengths two and three. After a small number of edges have been deleted, a path-lifting lemma turns such walks into simple paths of the original graph, and the original color classes then satisfy the compatibility relation.

An edge is \emph{triangular} if it lies in a triangle. Once the weighted minimum degree exceeds $1/3$, a set of pairwise compatible edges is either entirely triangular or entirely nontriangular, and in the first case it has at most two edges. The saving available in the triangular part is therefore the value $\nu$ of a capacitated fractional matching. In the nontriangular part, the geometry of neighborhoods gives a family of lower bounds $q(w)$ on the palette cost. The heart of the argument is the inequality
\[
 2\sum_wx_w\tau(w)q(w)
 +\left(\frac12-m-2f-2\nu\right)\sum_wx_w\tau(w)\ge0,
\]
where $m$ is the total edge mass, $f$ is the nontriangular edge mass, and $\tau(w)$ is the edge mass inside $N(w)$. We prove it with an exact rational certificate on five sampled vertices. The certificate uses two constraints supplied by the triangular matching: a bound on endpoint incidences in each neighborhood, and a bound on the union of the neighborhoods of each matched edge.

The resulting palette bound has a stable form that tolerates $m$ slightly below $1/4$. Some such stability is necessary. Path lifting relies on regularity and triangle removal and therefore deletes edges, while the surplus of the original graph over $n^2/4$ may be a single edge. When every bipartite cut misses many edges, a lower bound on the triangle mass pays for the deleted edges. When some cut contains almost all edges, we find a large clique in the conflict graph directly. Finally, a vertex-deletion argument removes the minimum-degree assumption while keeping the number of edges strictly above $h^2/4$, where $h$ is the current order.

Sections~\ref{sec:palettes}--\ref{sec:finite} prove the finite palette inequality. Sections~\ref{sec:cleanup}--\ref{sec:completion} transfer it to arbitrary graphs and complete the proof of Theorem~\ref{thm:main}. Section~\ref{sec:verification} describes the Lean formalization of Corollary~\ref{cor:allodd}, which combines the new $C_7$ proof with a formalization of the Buci\'c--Chen--Ma argument for $k\ge4$.

\section{Weighted palettes and private resources}\label{sec:palettes}
\subsection{Compatibility}
Throughout this section, $Q$ is a finite simple graph with positive vertex weights $(x_v)_{v\in V(Q)}$ satisfying $\sum_vx_v=1$. For $U\subseteq V(Q)$ and $A\subseteq E(Q)$, write
\[
 x(U)=\sum_{v\in U}x_v,\qquad c(uv)=x_ux_v,\qquad
 c(A)=\sum_{e\in A}c(e).
\]
Put $D(v)=x(N(v))$, $\delta=\min_vD(v)$, and $m=c(E(Q))$. Neighborhoods are open, and walks and neighborhoods are always taken in $Q$.

\begin{definition}\label{def:compatibility}
Two distinct edge types $ab,cd\in E(Q)$ are \emph{compatible} if
\begin{equation}\label{eq:compatibility}
 (A_Q^2)_{ac}(A_Q^3)_{bd}=0
\end{equation}
for every choice of orientations $(a,b)$ and $(c,d)$ of the two edges. Here $A_Q$ is the adjacency matrix, so its powers count walks; walks may repeat vertices and may be closed. A \emph{pattern} is a nonempty pairwise compatible set of edge types.
\end{definition}
Compatibility is defined through walks rather than simple cycles. Section~\ref{sec:source} shows that, after the cleanup of Section~\ref{sec:cleanup}, the original color classes give patterns in the retained graph.

Let $T$ be the set of edges of $Q$ that lie in a triangle, and put $F=E(Q)\setminus T$. Let $Z$ be the set of vertices lying in no triangle, and define
\begin{equation}\label{eq:sectors}
 S=E(Q[Z]),\qquad J=F\setminus S,\qquad f=c(F).
\end{equation}
Thus $S\subseteq F$, and $f$ includes the mass of $S$.
For $A\subseteq E(Q)$, an \emph{exact fractional pattern cover} of $A$ is a collection of amounts $a_p\ge0$, indexed by patterns contained in $A$, with
\begin{equation}\label{eq:exactcover}
 \sum_{p\ni e}a_p=c(e)\quad(e\in A).
\end{equation}
Its cost is $\sum_pa_p$. Write $\Phi$ for the minimum cost of a cover of $E(Q)\setminus S$ and $\Phi_F$ for the minimum cost of a cover of $J$. The cover by singletons is feasible and there are finitely many patterns, so both minima are attained. Compatibility is always computed in $Q$, even when the covered set is smaller.

Lower bounds on covers come from weak duality. If $y_e\ge0$ and
\begin{equation}\label{eq:wholepattern}
 \sum_{e\in p}y_e\le1\quad\text{for every pattern }p\subseteq A,
\end{equation}
then every exact cover of $A$ has cost at least $\sum_{e\in A}c(e)y_e$. To see this, multiply~\eqref{eq:exactcover} by $y_e$ and interchange the two finite sums.

\subsection{Neighborhood geometry}
Both sectors are controlled by the following lemma.
\begin{lemma}\label{lem:cross}
For compatible edges $uv$ and $ab$, the nonempty intersections between the two indexed pairs
\[
 \{N(u),N(v)\}\quad\text{and}\quad\{N(a),N(b)\}
\]
form a partial matching: each member of either pair meets at most one member of the other. If this matching is perfect and the edges are oriented so that $N(u)\cap N(a)$ and $N(v)\cap N(b)$ are nonempty, then the crossed intersections are empty and the aligned pairs $N(u),N(a)$ and $N(v),N(b)$ are anticomplete.
\end{lemma}
\begin{proof}
If $N(u)$ meets both $N(a)$ and $N(b)$, there is a two-walk from $u$ to $a$, and for $z\in N(u)\cap N(b)$ the walk $vuzb$ has length three. These complementary walks are forbidden by~\eqref{eq:compatibility}. Relabeling the endpoints and exchanging the two edges gives the remaining row and column restrictions. In the perfect case, a two-walk from $u$ to $a$ excludes every three-walk from $v$ to $b$; equivalently, there is no edge between $N(v)$ and $N(b)$. The other aligned pair is treated in the same way.
\end{proof}
Neighborhoods are indexed by endpoints, so they are counted separately even when two endpoints coincide. We shall use repeatedly that when $\delta>1/3$ no three vertex neighborhoods are pairwise disjoint, since their weights would sum to more than one.

\begin{lemma}[Sector separation]\label{lem:sectors}
Suppose $\delta>1/3$. No compatible pair mixes $T$ and $F$, and every compatible triple consists entirely of edges of $F$.
\end{lemma}
\begin{proof}
Suppose $uv$ is compatible with $ab\in F$. Since $N(a)$ and $N(b)$ are disjoint, neither $N(u)$ nor $N(v)$ can avoid both, for that would give three disjoint neighborhoods. By Lemma~\ref{lem:cross} the intersections therefore form a perfect matching; align the pairs as in that lemma. If $uv\in T$, choose $z\in N(u)\cap N(v)$. By anticompleteness $N(z)$ avoids both $N(a)$ and $N(b)$, again giving three disjoint neighborhoods. This proves the first assertion.

For the second, consider the six indexed endpoint neighborhoods of three compatible edges, grouped into three pairs. Between any two pairs the intersection relation is a partial matching. There is no independent transversal of the three pairs, because three selected neighborhoods with no pairwise intersections would be disjoint.

We claim that every cross matching is perfect. Suppose a member $A$ of the first pair meets neither member of the second. For each member $B$ of the second pair, every member of the third must meet $A$ or $B$, since otherwise these three form an independent transversal. Each of $A,B$ has at most one cross neighbor in the third pair. Hence $A$ meets one member of the third pair, and \emph{both} members of the second pair must meet the other. This contradicts the partial-matching property.

Align the second and third pairs with the first, and denote their members by $A_i,B_i$ for $i=1,2,3$. The matching between pairs two and three is also aligned, since a crossed matching would leave $A_1,B_2,B_3$ pairwise disjoint. Thus, for distinct $i,j$, the sets $A_i,B_j$ are disjoint, while $A_i,A_j$ and $B_i,B_j$ are anticomplete by Lemma~\ref{lem:cross}. If the $i$th edge is triangular, take $z\in A_i\cap B_i$. For the two other indices $j,k$, the three neighborhoods $N(z),A_j,B_k$ are disjoint, a contradiction. Consequently all three edges lie in $F$.
\end{proof}

Let $\mathcal C_T$ be the graph whose vertices are the edge types in $T$ and whose edges are the compatible pairs. A \emph{capacitated fractional matching} in $\mathcal C_T$ assigns $z_{eg}\ge0$ to each unordered compatible pair $\{e,g\}$, subject to
\[
 \sum_{g:\{e,g\}\in E(\mathcal C_T)}z_{eg}\le c(e)\quad(e\in T).
\]
Let $\nu$ be the maximum of $\sum_{\{e,g\}}z_{eg}$, the value of a bounded linear program in finitely many variables.

\begin{proposition}[Exact palette decomposition]\label{prop:decomposition}
If $\delta>1/3$, then
\begin{equation}\label{eq:decomposition}
 \Phi=\Phi_F+c(T)-\nu=\Phi_F+m-f-\nu.
\end{equation}
\end{proposition}
\begin{proof}
By Lemma~\ref{lem:sectors}, each pattern lies wholly in $F$ or wholly in $T$, and a $T$-pattern is a singleton or a pair. The pair amounts of any exact $T$-cover form a feasible fractional matching, and if their sum is $v$, counting the covered capacity shows that the cost is $c(T)-v$. Conversely, a feasible matching extends to an exact cover of the same cost by covering its unused capacities with singletons. Minimizing gives $c(T)-\nu$, independently of the $J$-sector.
\end{proof}

\subsection{Private resources}
For $v\in V(Q)$ write $N_T(v)=\{u:uv\in T\}$ and $d_F(v)=x(N_F(v))$. Define
\begin{equation}\label{eq:resources}
 R_{uv}=N_T(u)\cup N_T(v)\quad(uv\in F),\qquad
 q(w)=\sum_{v:wv\in T}x_vd_F(v).
\end{equation}

\begin{lemma}[Private-resource prices]\label{lem:resources}
Suppose $\delta>1/3$. If $uv,ab\in F$ are compatible, then
\[
 R_{uv}\cap\bigl(N(a)\cup N(b)\bigr)=\varnothing.
\]
Consequently, for every $w$, the set
\[
 D_w=\{e\in J:w\in R_e\}
\]
meets every pattern in $J$ at most once, and
\begin{equation}\label{eq:resourcebound}
 c(D_w)=q(w),\qquad \Phi_F\ge\max_wq(w).
\end{equation}
\end{lemma}
\begin{proof}
The proof of Lemma~\ref{lem:sectors} gives a perfect alignment for the two edges. In that alignment $N(u)$ avoids $N(b)$, so $N_T(u)$ avoids $N(b)$. If $z\in N_T(u)\cap N(a)$, the triangular edge $uz$ has a common neighbor $t\in N(u)$ adjacent to $z$, and the edge $tz$ contradicts anticompleteness of $N(u),N(a)$. Similarly, $N_T(v)$ avoids both neighborhoods of the mate. This proves the asserted disjointness, and in particular $R_{uv}\cap R_{ab}=\varnothing$.

Hence a fixed $w$ lies in the resource of at most one edge of any pattern, and the indicator of $D_w$ satisfies~\eqref{eq:wholepattern}. Edges in $S$ have empty resources. Moreover, $N(u)\cap N(v)=\varnothing$ for $uv\in F$, so the two possible resource incidences of such an edge cannot both occur. Expanding the objective therefore gives
\[
 \sum_{uv\in F}x_ux_v\ind{w\in R_{uv}}
 =\sum_{v:wv\in T}x_v\sum_{u:uv\in F}x_u=q(w).
\]
The bound on $\Phi_F$ now follows from~\eqref{eq:wholepattern}.
\end{proof}

\section{Two constraints from the triangular matching}\label{sec:matching}
Fix an optimal fractional matching $z$ in $\mathcal C_T$. Define its edge usage, marking probability, total usage, and marked degree by
\begin{align}
 \sigma_e&=\sum_{g:\{e,g\}\in E(\mathcal C_T)}z_{eg},
 &\lambda_{uv}&=\frac{\sigma_{uv}}{x_ux_v}\quad(uv\in T),\label{eq:marks}\\
 M&=\sum_{e\in T}\sigma_e=2\nu,
 &k(v)&=\sum_{u:uv\in T}x_u\lambda_{uv}.
 \label{eq:markeddegree}
\end{align}
Set $\lambda_{uv}=0$ off $T$. The capacity constraints give $0\le\lambda_{uv}\le1$, and
\[
 \sum_vx_vk(v)=2M.
\]
The marks satisfy two constraints, one for each vertex and one for each marked edge.

\begin{lemma}[Root-incidence budget]\label{lem:root}
For every $w\in V(Q)$,
\begin{equation}\label{eq:rootbudget}
 K(w):=M-\sum_{v\in N(w)}x_vk(v)\ge0.
\end{equation}
\end{lemma}
\begin{proof}
A compatible pair $e,g$ has at most two endpoint incidences in $N(w)$, where an endpoint is counted once for each edge containing it. Otherwise, after orienting the pair as $ab,cd$, the vertex $w$ is adjacent to $a,b,c$, and the walks $awc$ and $bwcd$ contradict compatibility. Hence
\[
 K(w)=\sum_{\{e,g\}}z_{eg}
 \bigl(2-\abs{e\cap N(w)}-\abs{g\cap N(w)}\bigr)\ge0,
\]
where the equality follows directly from~\eqref{eq:markeddegree}.
\end{proof}

\begin{lemma}[Selected-edge union bound]\label{lem:union}
Let $uv\in T$ have a compatible mate in $T$. There is a vertex $z$ such that
\[
 N(z)\cap\bigl(N(u)\cup N(v)\bigr)=\varnothing.
\]
Consequently,
\begin{equation}\label{eq:union}
 \lambda_{uv}>0\quad\Longrightarrow\quad
 x\bigl(N(u)\cup N(v)\bigr)\le1-\delta<\frac23
\end{equation}
when $\delta>1/3$.
\end{lemma}
\begin{proof}
Choose a compatible triangular mate $ab$. If the partial matching of cross-intersections in Lemma~\ref{lem:cross} has at most one entry, one of $N(a),N(b)$ avoids both $N(u)$ and $N(v)$, and we take the corresponding endpoint as $z$. Otherwise the matching is perfect; orient it so that the aligned pairs are $N(u),N(a)$ and $N(v),N(b)$. Since $ab$ is triangular, there is $z\in N(a)\cap N(b)$, and anticompleteness of the aligned pairs implies that $N(z)$ avoids $N(u)\cup N(v)$. In either case the union is disjoint from a neighborhood of mass at least $\delta$.

Finally, $\lambda_{uv}>0$ means $\sigma_{uv}>0$, so some pair containing $uv$ has positive matching amount and supplies a triangular mate.
\end{proof}

In the next section we set the matching aside and prove a polynomial inequality for arbitrary marks satisfying these two constraints.

\section{The exact polynomial certificate}\label{sec:certificate}
We now prove the weighted inequality behind the lower bound. The proof has two parts: a finite identity, one equation for each admissible colored graph on five vertices, verified by computer in exact rational arithmetic; and an averaging argument showing that every term of the identity contributes with the correct sign.

\subsection{The marked-host statement}
Let $Q$ again have positive vertex weights summing to one, but in this section allow any partition $E(Q)=F\sqcup T$ such that \emph{no triangle contains an edge of $F$}. Set $f=c(F)$ and define $d_F$ and $q$ by~\eqref{eq:resources}. Assign marks $0\le\lambda_e\le1$ to the edges of $T$, with $\lambda_e=0$ off $T$, and put
\[
 \sigma_{uv}=x_ux_v\lambda_{uv},\qquad
 M=\sum_{uv\in T}\sigma_{uv},\qquad
 k(v)=\sum_u x_u\lambda_{uv}.
\]
The marks need not come from a matching. Define
\begin{equation}\label{eq:triangleobservables}
 \tau(w)=c\bigl(E(Q[N(w)])\bigr),\qquad
 \Delta=\sum_wx_w\tau(w),\qquad
 B=\sum_wx_w\tau(w)q(w).
\end{equation}
Thus $\Delta$ is three times the weighted number of triangles, an unordered triangle $uvw$ having weight $x_ux_vx_w$.

\begin{theorem}[Exact marked-host inequality]\label{thm:polynomial}
Suppose $\delta\ge1/3$, the root budgets
\[
 M-\sum_{v\in N(w)}x_vk(v)\ge0\quad(w\in V(Q))
\]
hold, and
\[
 \lambda_{uv}>0\ \Longrightarrow\
 x(N(u)\cup N(v))\le2/3.
\]
Then, if $m\ge1/4$,
\begin{equation}\label{eq:polynomial}
 \cL:=2B+\left(\frac12-m-2f-M\right)\Delta\ge0.
\end{equation}
More generally, if $\xi\ge0$ and $m\ge1/4-\xi$, then
\begin{equation}\label{eq:polynomialstable}
 \cL\ge-a_*\xi\ge-\xi,
 \qquad a_*=\frac{383936867}{1000000000}<1.
\end{equation}
\end{theorem}
The other hypotheses are the same in both assertions; only the density condition is weakened.

\subsection{A five-position probability space}
Sample host vertices $X_0,\ldots,X_4$ independently according to $x$. Conditional on these types, color each unordered pair of \emph{positions} independently, according to
\[
 a_{ij}=\begin{cases}
 0,&X_iX_j\notin E(Q),\\
 1,&X_iX_j\in F,\\
 3\text{ with probability }\lambda_{X_iX_j},\ \ 2\text{ otherwise},
   &X_iX_j\in T.
 \end{cases}
\]
Host types may repeat, and $X_i=X_j$ gives $a_{ij}=0$. When two position pairs represent the same host edge, their marks are still independent. The root and union constraints will enter through conditional expectations in this probability space.

A colored graph is \emph{admissible} if no triangle has an edge of color~1; every sample is admissible. For an ordered sample, abbreviate
\[
 E_{ij}=\ind{a_{ij}>0},\quad F_{ij}=\ind{a_{ij}=1},\quad
 T_{ij}=\ind{a_{ij}\ge2},\quad L_{ij}=\ind{a_{ij}=3},
 \qquad t=E_{01}E_{02}E_{12},
\]
and define the objective contribution
\begin{equation}\label{eq:orderedobjective}
 o(a)=4tT_{03}F_{34}+t-tE_{34}-2tF_{34}-tL_{34}.
\end{equation}
The elementary identities
\begin{gather*}
 \E t=2\Delta,\qquad \E(tT_{03}F_{34})=2B,\\
 \E E_{34}=2m,\qquad \E F_{34}=2f,\qquad \E L_{34}=2M
\end{gather*}
give
\begin{equation}\label{eq:objectiveexpectation}
 \E o(a)=4\cL.
\end{equation}
For the mixed term, condition on $X_0=w$: the extensions on positions $1,2$ and on $3,4$ are independent, with expectations $2\tau(w)$ and $q(w)$. In the other products the triangle on $0,1,2$ is independent of the pair $3,4$. These computations remain valid when host types coincide.

\subsection{The finite identity}
We now describe the flags and matrices in the certificate. A flag is a colored graph with specified labeled vertices, and isomorphisms of flags fix the labels. Canonical representatives are chosen as described in Appendix~\ref{app:data}.

Let $\mathcal A$ and $\mathcal D$ be the admissible flags on three and four positions, respectively, with position~0 labeled. Let $\mathcal T_3$ be the admissible unrooted three-position types. Finally, let $\mathcal U$ be the admissible four-position flags with ordered labeled pair $(0,1)$ of color~3, so that only positions~2 and~3 may be exchanged. Then
\begin{equation}\label{eq:flagcounts}
 |\mathcal A|=28,\quad |\mathcal D|=294,\quad
 |\mathcal T_3|=14,\quad |\mathcal U|=190.
\end{equation}
For each $\theta\in\mathcal T_3$, fix its canonical labeled representative on positions $0,1,2$, and let $\mathcal V_\theta$ be the set of allowed attachment vectors $(a_{03},a_{13},a_{23})$.

Write $A_{ijk}$ for the flag in $\mathcal A$ induced on $(i,j,k)$ with root $i$, $D_{ijkl}$ for the corresponding flag in $\mathcal D$, and $\theta_{ijk}$ for the unrooted type. When $a_{01}=3$, write $U_{0123}$ for the ordered-edge-rooted flag. For the fixed labeled type on $0,1,2$, let $v_3,v_4$ denote the two attachment vectors.

The data consist of symmetric rational matrices $G_0$ on $\mathcal A$ and $G_\theta$ on $\mathcal V_\theta$, nonnegative rational arrays
\[
 (\alpha_\theta)_{\mathcal T_3},\qquad
 (\beta_D)_{\mathcal D},\qquad
 (\gamma_A)_{\mathcal A},\qquad
 (\eta_U)_{\mathcal U},
\]
and nonnegative rational slacks $s_H$, one for each admissible unlabeled five-position graph $H$. Set
\begin{align}
 g(a)={}&4G_0[A_{012},A_{034}]
 +4\sum_{\theta\in\mathcal T_3}
   \ind{a|_{012}=\theta}G_\theta[v_3,v_4],
 \label{eq:gramcolumn}\\
 d(a)={}&\alpha_{\theta_{012}}(2E_{34}-1),\label{eq:densitycolumn}\\
 h(a)={}&\frac43\beta_{D_{0123}}(3E_{04}-1),\label{eq:degreecolumn}\\
 r(a)={}&2\gamma_{A_{034}}L_{12}(1-E_{01}-E_{02}),\label{eq:rootcolumn}\\
 u(a)={}&\begin{cases}
 \displaystyle\frac43\eta_{U_{0123}}
 \bigl(2-3\ind{E_{04}=1\ \mathrm{or}\ E_{14}=1}\bigr),&a_{01}=3,\\
 0,&a_{01}\ne3.
 \end{cases}\label{eq:unioncolumn}
\end{align}
In~\eqref{eq:gramcolumn}, $a|_{012}=\theta$ denotes equality of labeled graphs, whereas the index $\theta_{012}$ in~\eqref{eq:densitycolumn} is the unrooted isomorphism type.

\begin{lemma}[Finite rational certificate]\label{lem:certificate}
The rational data described in Appendix~\ref{app:data} have the following properties. Every $G_i$ has a factorization $G_i=V_iR_iV_i^{\mathsf T}$ with $V_i$ integral and $R_i$ rational positive definite. All scalar multipliers and all $s_H$ are nonnegative, and $\max_\theta\alpha_\theta=a_*$. For every admissible unlabeled five-position graph $H$,
\begin{equation}\label{eq:rowidentity}
 \sum_{\pi\in S_5}o(H^\pi)
 =\sum_{\pi\in S_5}\bigl(g+d+h+r+u\bigr)(H^\pi)+s_H.
\end{equation}
There are exactly $117916$ admissible labeled graphs and $1436$ disjoint unlabeled orbits covering them.
\end{lemma}
\begin{proof}[Computer-assisted proof]
Enumerating the $4^{10}$ color assignments on five positions gives the admissible labeled graphs. The orbits of the chosen representatives under the $120$ permutations are pairwise disjoint and cover this set. The flag and attachment lists come from the same enumeration with the prescribed labels fixed.

For each rational matrix $R_i$, exact symmetric elimination produces strictly positive pivots, so $R_i$ is positive definite and $G_i=V_iR_iV_i^{\mathsf T}$ is positive semidefinite. Substituting these matrices and the rational scalar multipliers into~\eqref{eq:rowidentity} verifies the identity for each orbit, with nonnegative slacks. All scalar multipliers are nonnegative, and the largest density multiplier is $a_*$.

All of these computations are exact. They were carried out by an independently written verifier in integer and rational arithmetic, and they are also checked in the Lean development described in Section~\ref{sec:verification}. Table~\ref{tab:certificate} lists the sizes of the data.
\end{proof}

\begin{table}[tb]
\centering
\caption{Sizes of the rational certificate. Omitted scalar entries are zero.}\label{tab:certificate}
\begin{tabular}{lr}
\toprule
Object & Number\\
\midrule
Admissible labeled five-position graphs & $117916$\\
Unlabeled orbits and coefficient identities & $1436$\\
Positive semidefinite Gram blocks & $15$\\
Strictly positive rational elimination pivots & $456$\\
Positive scalar multipliers & $481$\\
\quad density / degree / root / union & $13\,/\,256\,/\,28\,/\,184$\\
Strictly positive / exactly zero slacks & $1397\,/\,39$\\
\bottomrule
\end{tabular}
\end{table}

\subsection{Averaging the identity}
\begin{proof}[Proof of Theorem~\ref{thm:polynomial}]
The sampling law is invariant under permutations of the positions. Average~\eqref{eq:rowidentity} over the distribution of unlabeled samples and divide by $480=120\cdot4$. By~\eqref{eq:objectiveexpectation}, the left side becomes $\cL$. We show that each term on the right has nonnegative expectation, except for the density term.

For the root Gram term, condition on $X_0=w$. The two rooted flags on $0,1,2$ and on $0,3,4$ are then independent and identically distributed. If $p(w)$ is their conditional probability vector, their normalized contribution is
\[
 \E_w\bigl[p(w)^{\mathsf T}G_0p(w)\bigr]\ge0.
\]
For a three-root term, condition on $X_0,X_1,X_2$ and on the three marks among these positions. The two new attachments are again conditionally independent with the same distribution, so the contribution is a nonnegative quadratic form multiplied by the probability of its labeled type.

Independence of $0,1,2$ from $3,4$ gives
\begin{equation}\label{eq:densityexpectation}
 \frac14\E d(a)
 =(m-1/4)\sum_{\theta\in\mathcal T_3}
       \alpha_\theta\Prob(\theta_{012}=\theta).
\end{equation}
The type probabilities sum to one. Since position~4 is fresh, the degree term becomes a sum of nonnegative flag densities multiplied by $D(X_0)-1/3$, and is therefore nonnegative.

For the root term, condition on $X_0=w$ and put
\[
 R(w)=\sum_{v\in N(w)}x_vk(v).
\]
Then $\E(L_{12}\mid X_0=w)=2M$ and
\[
 \E(L_{12}E_{01}\mid X_0=w)
 =\E(L_{12}E_{02}\mid X_0=w)=R(w),
\]
so the conditional expectation of $2L_{12}(1-E_{01}-E_{02})$ is $4K(w)$. The extension on $1,2$ is conditionally independent of the flag on $0,3,4$. After division by four,~\eqref{eq:rootcolumn} therefore contributes a nonnegative flag density times $K(w)$, averaged over $w$.

For the union term, condition on the four types and the marks that determine its flag. If the selected root event has positive probability, the underlying root edge has positive $\lambda$. The fresh fifth type lies in the union of the two full neighborhoods with probability equal to the weight of that union. The normalized conditional expectation is
\[
 \eta_{U_{0123}}\ind{a_{01}=3}
 \left(\frac23-x(N(X_0)\cup N(X_1))\right),
\]
which is nonnegative on the flag event.

Finally, if $p_H$ is the probability of orbit $H$, the normalized slack contribution is $\sum_Hp_Hs_H/480\ge0$. Here the sum over permutations counts multiplicities, including distinct permutations that yield the same labeled graph.

If $m\ge1/4$, then~\eqref{eq:densityexpectation} is nonnegative as well, which proves~\eqref{eq:polynomial}. If $m\ge1/4-\xi$, its value is at least $-a_*\xi$, because the type probabilities sum to one and $0\le\alpha_\theta\le a_*$. All other terms remain nonnegative, which proves~\eqref{eq:polynomialstable}.
\end{proof}

The certificate is a flag-algebra computation in the sense of Razborov~\cite{Razborov2007}, although we have phrased its interpretation directly in terms of finite sampling. The matrices were found by semidefinite programming and then replaced by exact rational factorizations; only these factorizations and the identity~\eqref{eq:rowidentity} enter the proof.

\section{Strict and stable finite palette bounds}\label{sec:finite}
We return to the triangular and nontriangular sectors of Section~\ref{sec:palettes} and assume $\delta>1/3$. The optimal matching of Section~\ref{sec:matching} has $M=2\nu$, and its marks satisfy the degree, root, and union hypotheses of Theorem~\ref{thm:polynomial}.

\begin{lemma}[Weighted triangle positivity]\label{lem:mantel}
If $m>1/4$, then $\Delta>0$.
\end{lemma}
\begin{proof}
If $\Delta=0$, then $Q$ is triangle-free, since all vertex weights are positive. Hence $D(u)+D(v)\le1$ on every edge, and summing with edge weights yields
\[
 \sum_vx_vD(v)^2
 =\sum_{uv\in E(Q)}x_ux_v\bigl(D(u)+D(v)\bigr)\le m.
\]
On the other hand, by the Cauchy--Schwarz inequality and $\sum_vx_vD(v)=2m$, the left side is at least $4m^2$. Thus $m\le1/4$, a contradiction.
\end{proof}

\begin{theorem}[Finite palette bound]\label{thm:finite}
Every positively weighted finite simple graph with $\delta>1/3$ and $m>1/4$ satisfies
\begin{equation}\label{eq:finite}
 \Phi\ge\frac32m-\frac14.
\end{equation}
\end{theorem}
\begin{proof}
By Lemma~\ref{lem:mantel}, $\Delta>0$. Theorem~\ref{thm:polynomial} and Lemma~\ref{lem:resources} give
\[
 \Phi_F\ge\max_wq(w)\ge\frac B\Delta
 \ge f+\nu+\frac m2-\frac14.
\]
Add $m-f-\nu$ and use Proposition~\ref{prop:decomposition}.
\end{proof}
The strict inequality $m>1/4$ cannot be dropped: a balanced complete bipartite weighted host has $m=1/4$ and $S=E(Q)$, so $\Phi=0$.

For a weighted host write
\[
 \MC(Q)=\max_{U\subseteq V(Q)}c\bigl(E(U,V(Q)\setminus U)\bigr),
 \qquad b(Q)=m-\MC(Q).
\]

\begin{lemma}[Triangle mass from cut deficit]\label{lem:cuttriangle}
For every weighted host,
\begin{equation}\label{eq:cuttriangle}
 \MC(Q)\ge4m^2-2\Delta,
 \qquad 2\Delta\ge b(Q)+4m(m-1/4).
\end{equation}
If $m\ge1/4-\xi$, $b(Q)\ge\rho>0$, and $0\le\xi\le\rho/2$, then $\Delta\ge\rho/4$.
\end{lemma}
\begin{proof}
The cut with one side $N(w)$ has mass
\[
 \sum_{v\in N(w)}x_vD(v)-2\tau(w).
\]
Averaging over $w$, the first term becomes $\sum_vx_vD(v)^2\ge4m^2$, which proves~\eqref{eq:cuttriangle}. If $m\ge1/4$, the last term of~\eqref{eq:cuttriangle} is nonnegative. Otherwise $4m\le1$ and $m-1/4\ge-\xi$, so $4m(m-1/4)\ge-\xi$. In both cases $2\Delta\ge\rho-\xi\ge\rho/2$.
\end{proof}

\begin{theorem}[Stable finite palette bound]\label{thm:stable}
Suppose $\delta>1/3$, $m\ge1/4-\xi$, $b(Q)\ge\rho>0$, and $0\le\xi\le\rho/2$. Then
\begin{equation}\label{eq:stable}
 \Phi\ge\frac32m-\frac14-\frac{2\xi}{\rho}.
\end{equation}
\end{theorem}
\begin{proof}
Lemma~\ref{lem:cuttriangle} gives $\Delta\ge\rho/4$. By the stable assertion of Theorem~\ref{thm:polynomial},
\[
 \frac B\Delta\ge f+\nu+\frac m2-\frac14-\frac{\xi}{2\Delta}
 \ge f+\nu+\frac m2-\frac14-\frac{2\xi}{\rho}.
\]
Now use $\Phi_F\ge B/\Delta$ and~\eqref{eq:decomposition} as before.
\end{proof}
The error term does not depend on the number of vertices of $Q$. This is what allows us to apply the bound after deleting a small number of edges.

\section{Robust paths in the original graph}\label{sec:cleanup}
Using the regularity and triangle-removal lemmas, we find a spanning subgraph whose short walks can be replaced by simple paths of the original graph, with the same endpoints and avoiding a bounded set of prescribed vertices. These two lemmas are the only nonelementary graph-theoretic tools in the proof of Theorem~\ref{thm:main}.

For disjoint nonempty vertex sets $A,B$, write $d(A,B)=e(A,B)/(|A||B|)$. The pair is $\epsilon$-regular if
\[
 |d(A',B')-d(A,B)|\le\epsilon
\]
whenever $A'\subseteq A$, $B'\subseteq B$, $|A'|\ge\epsilon|A|$, and $|B'|\ge\epsilon|B|$.

\begin{theorem}[Equitable regularity]\label{thm:regularity}
For every $\epsilon>0$ and positive integer $t_0$, there are $T$ and $n_0$ such that every graph on $n\ge n_0$ vertices has a partition
\[
 V_0\sqcup V_1\sqcup\cdots\sqcup V_t,
 \qquad t_0\le t\le T,
\]
where $|V_0|\le\epsilon n$, the other clusters have equal size, and at most $\epsilon t^2$ unordered cluster pairs are not $\epsilon$-regular.
\end{theorem}
This is the standard equitable form of Szemer\'edi's regularity lemma~\cite{Szemeredi1978}; see also~\cite{ConlonFox2012}.

\begin{theorem}[Triangle removal]\label{thm:removal}
For every $\alpha>0$ there is $\zeta>0$ such that a graph on $N$ vertices with at most $\zeta N^3$ triangles can be made triangle-free by deleting at most $\alpha N^2$ edges.
\end{theorem}
This is the triangle case of the graph removal lemma; see~\cite{Fox2011}. Only the qualitative forms of both results are needed.

\begin{lemma}[Robust path cleanup]\label{lem:cleanup}
Given $\eta>0$ and an integer $k\ge0$, for all sufficiently large $n$ every $n$-vertex graph $G$ has a spanning subgraph $G_0$ with
\[
 e(G)-e(G_0)\le\eta n^2
\]
such that the following holds. If distinct vertices $a,b$ are joined in $G_0$ by a walk of length $\ell\in\{2,3,5\}$, then for every $W\subseteq V(G)\setminus\{a,b\}$ with $|W|\le k$ there is a simple $a$--$b$ path of length exactly $\ell$ in the \emph{original graph $G$}, avoiding $W$.
\end{lemma}
\begin{proof}
We construct one cleanup for lengths three and five and another for length two. Each property survives further edge deletion, so the intersection of the two subgraphs has both.

\emph{Lengths three and five.}
Choose $d>0$ small and $t_0$ large, and then choose $\epsilon>0$ so small that
\[
 \epsilon<d/10,\qquad 4\epsilon+d+1/t_0<\eta/2.
\]
Apply Theorem~\ref{thm:regularity}, and let $L$ be the common cluster size. Delete edges incident to $V_0$, edges inside a cluster, edges in irregular pairs, and edges in regular pairs of density less than $d$. In each remaining regular pair $(A,B)$, also delete every edge incident to a vertex with fewer than $(d-\epsilon)L$ \emph{original} neighbors in the opposite cluster. There are fewer than $\epsilon L$ such vertices on either side, since a set of $\epsilon L$ of them would violate regularity together with the whole opposite cluster.

The respective edge losses are at most
\[
 \epsilon n^2,\quad \frac{n^2}{2t_0},\quad
 \epsilon n^2,\quad \frac d2 n^2,\quad \epsilon n^2,
\]
up to rounding errors that are negligible for large $n$. Their sum is less than $\eta n^2/2$. Every retained edge joins two clusters forming a regular pair of density at least $d$ in $G$, and each of its endpoints has at least $(d-\epsilon)L$ neighbors in the opposite cluster.

We use the following immediate consequence of regularity. If $(A,B)$ is $\epsilon$-regular with density $p$ and $Y\subseteq B$ has size at least $\epsilon|B|$, then fewer than $\epsilon|A|$ vertices of $A$ have fewer than $(p-\epsilon)|Y|$ neighbors in $Y$; otherwise these vertices and $Y$ would violate regularity.

For a retained three-walk $av_1v_2b$, let $C_1,C_2$ be the clusters of $v_1,v_2$. The sets
\[
 L_1=N_G(a)\cap C_1,\qquad L_2=N_G(b)\cap C_2
\]
have size at least $(d-\epsilon)L$. Remove $W\cup\{a,b\}$ from them. For large $n$ the remaining sets still have size at least $\epsilon L$, so regularity of $(C_1,C_2)$ supplies an edge $xy$ between them. The path $axyb$ is simple, has length three, and avoids $W$.

For a retained five-walk $av_1v_2v_3v_4b$, let $C_i$ be the cluster of $v_i$, and put $L_1=N_G(a)\cap C_1$ and $L_4=N_G(b)\cap C_4$, both of size at least $(d-\epsilon)L$. All but $\epsilon L$ vertices of $C_2$ have at least $(d-\epsilon)|L_1|$ neighbors in $L_1$, and all but $\epsilon L$ vertices of $C_3$ have at least $(d-\epsilon)|L_4|$ neighbors in $L_4$. After $W\cup\{a,b\}$ is excluded, these two sets of good vertices are still large enough for regularity of $(C_2,C_3)$ to give an edge $yz$ between them. Choose
\[
 x\in N_G(y)\cap L_1,\qquad z'\in N_G(z)\cap L_4,
\]
at each step excluding $W$ and every vertex already chosen. Each available neighbor set initially has size at least $(d-\epsilon)^2L$, which exceeds $k+6$ for large $n$. Then $axyzz'b$ is the required simple five-path. Consecutive clusters are distinct. Nonconsecutive clusters, and vertices of the input walk, may coincide, but the explicit exclusions keep the vertices of the new path distinct.

\emph{Length two.}
Take three disjoint copies $A,B,C$ of $V(G)$. Between $A$ and $B$, and between $B$ and $C$, put copies of the adjacency relation of $G$. Add a virtual edge $u_Av_C$ precisely when $u\ne v$ and the original codegree $|N_G(u)\cap N_G(v)|$ is at most $k$. Each virtual edge lies in at most $k$ auxiliary triangles, and every auxiliary triangle contains a virtual edge, so there are at most $kn^2=o((3n)^3)$ auxiliary triangles.

By Theorem~\ref{thm:removal}, for large $n$ these triangles can be destroyed by deleting at most $\eta n^2/(2(k+1))$ auxiliary edges. When an $AB$ or $BC$ edge is deleted, delete its underlying edge of $G$. For each deleted virtual edge $u_Av_C$, delete $uw$ for every $w\in N_G(u)\cap N_G(v)$, at a cost of at most $k$ edges of $G$. The total loss is at most $\eta n^2/2$.

No retained two-walk $uwv$ with distinct endpoints can have original codegree at most $k$. Such a walk would form an auxiliary triangle $u_Aw_Bv_C$. If an $AB$ or $BC$ edge of that triangle was deleted, its underlying edge of $G$ was deleted; if its virtual edge was deleted, the edge $uw$ was deleted explicitly. Either way the walk was not retained. Thus the endpoints have at least $k+1$ common neighbors in $G$, and one of them lies outside $W$. It is not an endpoint because $G$ is simple, so it gives the required two-path.
\end{proof}
The threshold on $n$ in Lemma~\ref{lem:cleanup} depends only on $\eta$ and $k$.

\begin{lemma}[Degree pruning]\label{lem:prune}
Suppose $\delta(G)\ge(1/3+\gamma)n$, where $\gamma>0$. For every sufficiently small fixed $\eta>0$, the cleanup above can be followed by vertex deletion to give a subgraph $H$ of order $h$ such that, writing $\beta=\gamma/4$ and $\ell=e(G)-e(H)$,
\begin{align}
 h&\ge(1-2\eta/\beta)n,\label{eq:pruneorder}\\
 \ell&\le(\eta+2\eta/\beta)n^2,\label{eq:pruneloss}\\
 \delta(H)&>(1/3+\gamma/2)h.\label{eq:prunedegree}
\end{align}
The subgraph $H$ retains the path property of Lemma~\ref{lem:cleanup}, with the paths taken in $G$.
\end{lemma}
\begin{proof}
Delete the vertices that lost more than $\beta n$ incident edges in the cleanup; there are at most $2\eta n/\beta$ of them. The remaining induced subgraph of $G_0$ satisfies the bounds on order and total edge loss above, where the loss includes all edges at removed vertices. Its minimum degree is at least
\[
 (1/3+\gamma-\beta-2\eta/\beta)n.
\]
Choose $\eta<\beta\gamma/8$. Then this exceeds $(1/3+\gamma/2)n$, and hence $(1/3+\gamma/2)h$. Retained walks are still walks of $G_0$, and their lifted paths need only lie in $G$, not in $H$.
\end{proof}

\section{From colorings to the finite palette bound}\label{sec:source}
Write $r_7(G)=r_{C_7}(G)$. A graph of order $n$ is \emph{eligible} if $e(G)\ge\lfloor n^2/4\rfloor+1$, or equivalently $e(G)>n^2/4$. For an ordinary graph put
\[
 b(G)=e(G)-\max_{U\subseteq V(G)}e(U,V(G)\setminus U),
 \qquad P(n,e)=\frac32e-\frac{n^2}{4}.
\]
For a graph on $h$ vertices with uniform weights, the weighted deficit is $b(G)/h^2$.

\begin{lemma}[Color classes give compatible patterns]\label{lem:sourcepatterns}
Let $G$ have a valid rainbow-$C_7$ coloring using $r$ colors. Let $H\subseteq G$ have $h>0$ vertices and satisfy the path property of Lemma~\ref{lem:cleanup} with $k=5$, with all lifted paths in $G$. Let $Z$ be the set of vertices of $H$ lying in no $H$-triangle and $S=E(H[Z])$. Then every nonempty restriction of an original color class to $E(H)\setminus S$ is a matching and a pattern, with compatibility computed in $H$. With uniform vertex weights $1/h$,
\begin{equation}\label{eq:sourcecover}
 r/h^2\ge\Phi(H).
\end{equation}
\end{lemma}
\begin{proof}
Suppose two retained edges $uv,uw$ have the same color, with $v\ne w$. If $u$ lies on an $H$-triangle $uab$, there is a retained five-walk $vuabuw$. If instead $v$ or $w$ lies on a triangle, the corresponding five-walk is $vabvuw$ or $vuwabw$, respectively. In each case the endpoints are the distinct vertices $v,w$. Lift the walk to a simple path of length five in $G$ avoiding $u$. Together with $uv$ and $uw$ it forms a simple $C_7$ with two edges of the same color, a contradiction. Hence every monochromatic retained wedge has all three vertices in $Z$, and the restriction of a color class to $E(H)\setminus S$ is a matching.

Take two edges $ab,cd$ of this matching. If they are incompatible for some orientation, there is a two-walk from $a$ to $c$ and a three-walk from $b$ to $d$ in $H$. Lift the first walk to a path in $G$ avoiding $b,d$, and then lift the second to a path in $G$ avoiding the three vertices of the first. The two paths are vertex-disjoint, and together with $ab$ and $cd$ they form a simple seven-cycle containing both edges of the same color, again a contradiction. This proves compatibility in every orientation.

Assign amount $1/h^2$ to each nonempty restricted color class, adding the amounts when two classes give the same pattern. Every edge outside $S$ belongs to exactly one original color class, so its capacity $1/h^2$ is covered exactly. The total cost is at most $r/h^2$, which proves~\eqref{eq:sourcecover}.
\end{proof}
The coloring of $G$ itself is never modified; the edges of $S$ are left out only when the fractional cover is formed.

\begin{proposition}[The far-cut estimate]\label{prop:farcut}
For every $\gamma,\rho,t>0$, all sufficiently large eligible graphs $G$ on $n$ vertices with
\[
 \delta(G)\ge(1/3+\gamma)n,\qquad b(G)\ge\rho n^2
\]
satisfy
\begin{equation}\label{eq:farcut}
 r_7(G)\ge P(n,e(G))-tn^2.
\end{equation}
\end{proposition}
\begin{proof}
Fix a valid coloring with $r$ colors. Apply Lemmas~\ref{lem:cleanup} and~\ref{lem:prune} with $k=5$ and with $\eta>0$ to be chosen in terms of $\gamma,\rho,t$. Let $H$ be the resulting subgraph, of order $h$, and let $\ell=e(G)-e(H)$, which counts every lost edge. Give the vertices of $H$ uniform weight $1/h$ and put $\xi=\ell/h^2$. Eligibility gives
\[
 m_H=\frac{e(H)}{h^2}\ge\frac14-\xi.
\]
Every cut of $H$ extends to a cut of $G$, so $b(H)\ge b(G)-\ell$. Choose $\eta$ small enough that
\[
 \ell\le\rho n^2/2,\qquad \xi\le\rho/4,
\]
which is possible by~\eqref{eq:pruneorder}--\eqref{eq:pruneloss}. Then the normalized deficit of $H$ is at least $\rho/2$, and its weighted minimum degree is greater than $1/3$. Applying Theorem~\ref{thm:stable} with deficit parameter $\rho/2$, and then Lemma~\ref{lem:sourcepatterns}, we obtain
\[
 r\ge P(h,e(H))-\frac{4\ell}{\rho}.
\]
Since
\[
 P(n,e(G))-P(h,e(H))
 =\frac32\ell-\frac{n^2-h^2}{4}\le\frac32\ell,
\]
it follows that
\begin{equation}\label{eq:sourceerror}
 r\ge P(n,e(G))-\left(\frac32+\frac4\rho\right)\ell.
\end{equation}
Decrease $\eta$ so that the last error term is at most $tn^2$, and only then take $n$ sufficiently large. These choices depend only on $\gamma,\rho,t$, and not on $G$ or its coloring.
\end{proof}

\section{The near-cut case}\label{sec:nearcut}
For graphs close to a bipartite cut we bound the number of colors directly. The strict inequality $e(G)>n^2/4$ provides an internal edge, through which we route seven-cycles containing two prescribed crossing edges.

\begin{lemma}[Near-cut seven-cycle clique]\label{lem:nearcut}
Let $0<\varepsilon\le1/16$ and $n\ge64\varepsilon^{-3}$. Suppose $G$ is eligible and has a cut with at least $(1/4-\varepsilon^3/8)n^2$ edges. Then every rainbow-$C_7$ coloring of $G$ uses at least $(1/8-\varepsilon)n^2$ colors.
\end{lemma}
\begin{proof}
Put $\tau=\varepsilon/2$, so that $\tau\le1/32$ and $n\ge8\tau^{-3}$. Take a maximum cut $(A,B)$, with $|A|=a$, $|B|=b$, and $M_0$ crossing edges. Let $D=ab-M_0$ be the number of missing crossing edges and $I=e(G)-M_0$ the number of internal edges. The cut assumption and eligibility imply
\begin{equation}\label{eq:nearcutbasic}
 |a-n/2|,|b-n/2|\le\tau^{3/2}n,\qquad
 D\le\tau^3n^2,\qquad I>D,
\end{equation}
where the last inequality holds because $I>n^2/4-M_0\ge D$.

Call a vertex \emph{typical} if it misses at most $\tau n$ vertices on the opposite side. Let $W$ be the set of remaining vertices and $s=|W|$. Counting the two endpoints of each missing edge gives
\begin{equation}\label{eq:badvertices}
 s\le2\tau^2n.
\end{equation}
We first find an internal edge $uv$, say in $B$, such that $u$ is typical and
\begin{equation}\label{eq:trigger}
 d_A(v)>a/2-s.
\end{equation}
Suppose no such edge exists on either side. Then no internal edge has two typical endpoints, since by~\eqref{eq:nearcutbasic}--\eqref{eq:badvertices} a typical vertex has crossing degree greater than half the opposite side minus $s$.

For $v\in W$, write $d_{\rm int}(v)$ and $d_{\rm cr}(v)$ for its internal and crossing degrees and $\operatorname{def}(v)$ for its number of missing crossing neighbors. Local optimality of a maximum cut gives $d_{\rm int}(v)\le d_{\rm cr}(v)$. If $v$ has a typical internal neighbor, the assumed failure of~\eqref{eq:trigger} gives
\[
 d_{\rm cr}(v)\le |\text{opposite side}|/2-s,
 \qquad \operatorname{def}(v)-d_{\rm int}(v)\ge2s.
\]
If it has no typical internal neighbor, then $d_{\rm int}(v)\le s-1$, whereas $\operatorname{def}(v)>\tau n\ge16s$, which gives the same inequality. Every internal edge meets $W$. If $s>0$, it follows that
\begin{align*}
 I&\le\sum_{v\in W}d_{\rm int}(v)
 \le\sum_{v\in W}\operatorname{def}(v)-2s^2\\
 &\le D+|W\cap A||W\cap B|-2s^2
 \le D-\frac74s^2<D,
\end{align*}
contradicting~\eqref{eq:nearcutbasic}; in the third step, a missing crossing edge with both endpoints in $W$ is the only way the deficit sum can exceed $D$. If $s=0$, the absence of internal edges between typical vertices gives $I=0$, which is also impossible. Thus the required internal edge $uv$ exists.

Set
\[
 X=(N_A(u)\cap N_A(v))\setminus W,
 \qquad B_0=B\setminus(W\cup\{u,v\}).
\]
Using typicality of $u$,~\eqref{eq:trigger}, and the preceding bounds, we obtain
\begin{align}
 |X|&>\left(\frac14-\tau-\frac12\tau^{3/2}-4\tau^2\right)n
 \ge\left(\frac14-\frac32\tau\right)n,\label{eq:nearX}\\
 |B_0|-\tau n
 &\ge\left(\frac12-\tau-\tau^{3/2}-2\tau^2\right)n-2
 \ge\left(\frac12-\frac32\tau\right)n.\label{eq:nearB}
\end{align}
The final inequalities use
\[
 \frac{\sqrt\tau}{2}+4\tau\le\frac12,
 \qquad \sqrt\tau+2\tau+\frac{2}{\tau n}\le\frac12,
\]
which hold in the stated range. All vertices of $X$ are typical, so the set $E(X,B_0)$ of selected edges has size at least
\begin{equation}\label{eq:nearcount}
 |X|(|B_0|-\tau n)
 \ge\left(\frac18-\frac98\tau\right)n^2
 \ge\left(\frac18-\varepsilon\right)n^2.
\end{equation}
We show that any two distinct selected edges lie on a common simple $C_7$.

First take disjoint edges $xy,zw$, where $x,z\in X$ and $y,w\in B_0$. Choose $c\in N_A(y)\cap N_A(w)$ outside $\{x,z\}$, and use the cycle
\begin{equation}\label{eq:cycle-disjoint}
 v\,x\,y\,c\,w\,z\,u\,v.
\end{equation}
For two selected edges $xy,xw$ sharing their endpoint in $A$, choose $c\in N_A(v)\cap N_A(y)$ outside $\{x\}$ and then $d\in N_A(u)\cap N_A(w)$ outside $\{x,c\}$, and use
\begin{equation}\label{eq:cycle-shareA}
 v\,c\,y\,x\,w\,d\,u\,v.
\end{equation}
For two edges $xy,zy$ sharing their endpoint in $B$, choose $w\in N(z)\cap B_0$ outside $\{y\}$ and $c\in N_A(w)\cap N_A(u)$ outside $\{x,z\}$, and use
\begin{equation}\label{eq:cycle-shareB}
 v\,x\,y\,z\,w\,c\,u\,v.
\end{equation}
These choices are possible. Two typical vertices of $B$ have at least $a-2\tau n>2$ common neighbors in $A$. The neighborhood of a typical vertex of $B$ meets that of either endpoint of $uv$ in more than $a/2-s-\tau n>2$ vertices. Finally, $|B_0|-\tau n>1$. These inequalities follow from~\eqref{eq:nearcutbasic}--\eqref{eq:nearB} and $n\ge8\tau^{-3}$.

Each displayed cycle has three distinct vertices in $A$ and four distinct vertices in $B$; the latter are distinct because $B_0$ excludes $u,v$ and each new choice excludes the vertices already selected. Thus the cycles are simple. The selected edges therefore form a clique in the conflict graph and receive distinct colors, and~\eqref{eq:nearcount} completes the proof.
\end{proof}

\begin{proposition}[Potential bound at large minimum degree]\label{prop:terminal}
For every $\gamma,t>0$ there is $n_0$ such that every eligible graph $G$ of order $n\ge n_0$ with $\delta(G)\ge(1/3+\gamma)n$ satisfies
\begin{equation}\label{eq:terminal}
 r_7(G)\ge P(n,e(G))-tn^2.
\end{equation}
\end{proposition}
\begin{proof}
Choose $0<\varepsilon_0\le1/16$ and $\rho>0$ with
\[
 \rho\le\varepsilon_0^3/8,\qquad
 \varepsilon_0+3\rho/2\le t.
\]
If $b(G)\ge\rho n^2$, apply Proposition~\ref{prop:farcut}. Otherwise, eligibility gives a maximum cut of size at least $(1/4-\rho)n^2$, while
\[
 e(G)\le(1/4+\rho)n^2,\qquad
 P(n,e(G))\le(1/8+3\rho/2)n^2.
\]
Lemma~\ref{lem:nearcut} then yields
\[
 r_7(G)\ge(1/8-\varepsilon_0)n^2
 \ge P(n,e(G))-(\varepsilon_0+3\rho/2)n^2.
\]
Take $n_0$ large enough for both cases. The constants depend only on $\gamma$ and $t$.
\end{proof}

\section{Peeling and the upper bound}\label{sec:completion}
\subsection{Preserving the exact threshold}
We now remove the minimum-degree hypothesis without giving up the surplus of one edge.

\begin{lemma}[Fixed-margin peeling]\label{lem:peeling}
Let $n\ge64$ and $0<\gamma\le1/48$. Suppose $G$ has exactly $\lfloor n^2/4\rfloor+1$ edges. Repeatedly delete a vertex of current degree less than $(1/3+\gamma)N+1$, where $N$ is the current order. The process stops at an induced subgraph $H$ of order $h>n/3$ such that
\begin{gather}
 e(H)>h^2/4,\qquad \delta(H)\ge(1/3+\gamma)h+1,\label{eq:peelterminal}\\
 P(h,e(H))\ge\frac{n^2}{8}
 -\frac{3\gamma}{4}n(n+1)-\frac{7n}{4}.\label{eq:peelpotential}
\end{gather}
\end{lemma}
\begin{proof}
Deleting a vertex of degree $d<(1/3+\gamma)N+1$ changes the potential by
\begin{equation}\label{eq:peelchange}
 P(N-1,m-d)-P(N,m)
 =\frac{2N-1}{4}-\frac32d
 >-\frac{3\gamma}{2}N-\frac74.
\end{equation}
The initial potential exceeds $n^2/8$. Summing~\eqref{eq:peelchange} and using $\sum N\le n(n+1)/2$ proves~\eqref{eq:peelpotential} for every initial segment of the deletion process.

If the order first reached some $h\le n/3$, the trivial bound $e(H)\le h^2/2$ would give
\[
 P(h,e(H))\le h^2/2\le n^2/18.
\]
On the other hand,~\eqref{eq:peelpotential} and $\gamma\le1/48$ give
\[
 P(h,e(H))\ge\frac{7}{64}n^2-\frac{113}{64}n>n^2/18
 \quad(n\ge64),
\]
a contradiction. Therefore every deletion is made at $N>n/3$, and the process stops at an order above $n/3$.

At each such deletion the surplus $m-N^2/4$ increases, since
\begin{align*}
 \bigl[(m-d)-(N-1)^2/4\bigr]-\bigl[m-N^2/4\bigr]
 &=\frac{2N-1}{4}-d\\
 &>\left(\frac16-\gamma\right)N-\frac54>0.
\end{align*}
The last inequality follows from $n\ge64$ and $\gamma\le1/48$. The initial surplus is positive, so $e(H)>h^2/4$. The stopping rule gives the asserted minimum degree.
\end{proof}

\begin{proof}[Lower bound in Theorem~\ref{thm:main}]
It suffices to consider $0<\varepsilon\le1/16$. Fix an eligible graph $G$ on $n$ vertices and a valid coloring. First keep exactly $\lfloor n^2/4\rfloor+1$ edges and restrict the coloring to them. Apply Lemma~\ref{lem:peeling} with $\gamma=\varepsilon/12$. The terminal graph $H$ has $h>n/3$, so for sufficiently large $n$ Proposition~\ref{prop:terminal} applies to it with $t=\varepsilon/2$. Since restriction cannot increase the number of colors,
\begin{align*}
 r_7(G)&\ge r_7(H)
 \ge P(h,e(H))-\frac{\varepsilon}{2}h^2\\
 &\ge\frac{n^2}{8}
 -\frac{\varepsilon}{16}n(n+1)-\frac{7n}{4}
 -\frac{\varepsilon}{2}n^2
 \ge\left(\frac18-\varepsilon\right)n^2
\end{align*}
for all sufficiently large $n$. The threshold on $n$ depends only on $\varepsilon$, so the bound holds for all eligible graphs and all valid colorings.
\end{proof}

\subsection{The two-clique upper bound}
The two-clique construction behind the conjecture gives the matching upper bound~\cite{BEGS1989,BCM2026}.

\begin{proof}[Upper bound in Theorem~\ref{thm:main}]
Put
\[
 a=\left\lceil\frac n2+\sqrt n\right\rceil,\qquad b=n-a,
\]
and take $G=K_a\sqcup K_b$. For sufficiently large $n$ these are nonnegative integers with $a\ge b$, and
\begin{align*}
 e(G)&=\binom a2+\binom b2
 =\frac{n^2}{4}+\left(a-\frac n2\right)^2-\frac n2\\
 &\ge\frac{n^2}{4}+\frac n2
 \ge\left\lfloor\frac{n^2}{4}\right\rfloor+1.
\end{align*}
Give the edges of each clique distinct colors, reusing colors of the larger clique on the smaller one. Every cycle lies in a single component and is therefore rainbow. The number of colors is at most
\[
 \binom a2=\frac{n^2}{8}+O(n^{3/2})
 =\left(\frac18+o(1)\right)n^2.
\]
Together with the lower bound, this proves Theorem~\ref{thm:main}, and hence Corollary~\ref{cor:allodd}.
\end{proof}

\section{Verification and formalization}\label{sec:verification}
The only computer-assisted step in the proof of Theorem~\ref{thm:main} is Lemma~\ref{lem:certificate}. Everything else, including the interpretation of the certificate and the passage to arbitrary graphs, is proved in the preceding sections. Separately, the whole of Corollary~\ref{cor:allodd} has been formalized in Lean~4.

\subsection{The rational certificate}
The certificate was found by semidefinite optimization and then reconstructed over the rationals. The independent verifier enumerates the admissible graphs and flags from scratch, forms the Gram matrices from their rational factorizations, and checks all $1436$ coefficient identities~\eqref{eq:rowidentity}. Exact elimination gives $456$ positive pivots in the reduced matrices. The verifier also confirms that the multipliers and slacks are nonnegative and that $\max_\theta\alpha_\theta=383936867/10^9<1$, the bound used in Theorem~\ref{thm:stable}. Together with the averaging argument of Section~\ref{sec:certificate}, these finite checks prove Theorem~\ref{thm:polynomial}. The verifier and the certificate data are in the directory \code{certificate} of the accompanying repository \url{https://github.com/Asad-Shahab/erdos-809-lean}, and Appendix~\ref{app:data} describes the data further.

\subsection{The Lean development}
The formalization is written in Lean~4.28.0~\cite{Lean4} against mathlib~v4.28.0~\cite{Mathlib}, and is contained in the same repository; the description below refers to commit \code{7ec4aaa}. Its three main theorems are the following.
\begin{description}[leftmargin=1.5em,itemsep=3pt]
\item[\code{Erdos809.erdos_809_C7}] Theorem~\ref{thm:main}. Its proof includes the finite certificate and its interpretation, the reduction from colorings to the palette bound, and the upper bound. It is stated in \code{Erdos809/Main.lean}.
\item[\code{Erdos809.erdos_809_long_odd_cycles}] The case $k\ge4$ of Corollary~\ref{cor:allodd}, following~\cite[Sections~3--4]{BCM2026}. The mathematics of this part is due to Buci\'c, Chen, and Ma; the contribution here is its formalization. It is stated in \code{Erdos809/LongOddCycles/Main.lean}.
\item[\code{Erdos809.erdos_809}] Corollary~\ref{cor:allodd}, deduced from the two previous theorems by separating the case $k=3$ from $k\ge4$. It is stated in \code{Erdos809/OddCycles.lean}.
\end{description}

The last theorem states that for every integer $k\ge3$ and real $\varepsilon>0$ there is $N$ such that, for every integer $n\ge N$, the minimum $q$ in~\eqref{eq:fdefinition} with $H=C_{2k+1}$ and $e=\lfloor n^2/4\rfloor+1$ is attained and satisfies $(1/8-\varepsilon)n^2\le q\le(1/8+\varepsilon)n^2$. Here $N$ may depend on $k$ and $\varepsilon$. A copy of the cycle is an injective adjacency-preserving map from $C_{2k+1}$ into the host graph, so copies are simple cycles that need not be induced, and the number of colors of a coloring is the size of its image. The minimum ranges over all graphs with at least $e$ edges and all valid colorings; a separate lemma shows that it is also attained by a graph with exactly $e$ edges.

In the $C_7$ part, the certificate appears as explicit data in the directory \code{Erdos809/Certificate}, scaled to integers by the common denominator given in Appendix~\ref{app:data}. These data comprise the enumeration of admissible five-vertex graphs with their orbit representatives, the fifteen Gram blocks together with integer witnesses of their positive semidefiniteness, and the coefficient identities. All of these finite statements are checked by the Lean kernel. The weighted and asymptotic bounds are then derived from them as in Sections~\ref{sec:certificate}--\ref{sec:completion}, with the regularity lemma and the triangle removal lemma taken from mathlib. The final theorem depends only on the standard axioms \code{propext}, \code{Classical.choice}, and \code{Quot.sound}; in particular, no step relies on compiled code or on an external computation.

The formalization differs from the exposition above in two inessential respects. It extracts a triangular matching from each exact cover, and its two-clique construction uses clique sizes differing by an $\varepsilon$-dependent linear amount. These variants give the same palette lower bound and the same asymptotic upper bound as the arguments presented here.

\subsection*{Acknowledgements}
The research for this paper made extensive use of AI systems, principally agents built on OpenAI's Codex and Anthropic's Claude. They were used in the search for the proof, the construction of the rational certificate, the Lean formalization, and the preparation of the manuscript. One finite counting lemma in the formalization of the Buci\'c--Chen--Ma argument was proved with the help of Harmonic's Aristotle. The results stated in Theorem~\ref{thm:main} and Corollary~\ref{cor:allodd} are checked by the Lean development described in Section~\ref{sec:verification}. The author is responsible for the content of the paper.

\clearpage
\appendix
\section{The rational certificate data}\label{app:data}
The certificate of Lemma~\ref{lem:certificate} consists of representatives of the admissible five-position graphs, the flag and attachment lists, fifteen factorizations $G_i=V_iR_iV_i^{\mathsf T}$, four families of scalar multipliers, and one slack for each representative. The matrices $V_i$ have integer entries, and the entries of $R_i$, the multipliers, and the slacks are rational.

Canonical representatives are defined as follows. A colored graph on $s$ ordered positions is encoded by its $\binom{s}{2}$ colors in lexicographic order of pairs, and we take the lexicographically least encoding among the permitted relabelings: all permutations for unrooted types, permutations fixing position~0 for one-root flags, and permutations fixing positions~0 and~1 individually for ordered-edge-rooted flags. Attachment vectors keep the fixed order of their three roots. These conventions determine the indices in~\eqref{eq:gramcolumn}--\eqref{eq:unioncolumn}.

In the order used by the data, the full Gram dimensions are
\[
 28,64,28,44,44,19,23,23,35,35,35,30,30,30,30,
\]
and the dimensions of the reduced positive definite matrices are
\[
 20,55,28,44,40,14,18,23,30,35,34,29,30,30,26,
\]
which sum to $456$. Among the $1436$ coefficient slacks, $39$ are zero and the smallest positive slack is $2835633/500000000$. The common scale $7461504000000000$ clears the denominators of the Gram matrices and scalar multipliers, as well as the divisions by three in the degree and union terms. Matrix products sum over ordered pairs of flags, as in $p^{\mathsf T}G_ip$.

The verifier reconstructs the admissible graphs and their orbit partition, the flag indices, and every coefficient in~\eqref{eq:rowidentity}. It checks positive definiteness of each $R_i$ by rational elimination, substitutes the resulting $G_i$ into the identity, and checks the signs of the multipliers and slacks and the exact value of $a_*$. The same data, scaled to integers, form part of the Lean development described in Section~\ref{sec:verification}.

\clearpage
\bibliographystyle{amsplain}
\bibliography{references}
\end{document}